\documentclass{amsart}
\usepackage{amscd,amssymb,amsmath,amsthm}
\usepackage[dvips]{graphicx}
\usepackage[all]{xy}
\theoremstyle{plain}
\newtheorem{proposition}{Proposition}
\newtheorem{theorem}[proposition]{Theorem}
\newtheorem{corollary}[proposition]{Corollary}
\newtheorem{lemma}[proposition]{Lemma}

\newtheorem*{proposition*}{Proposition}
\newtheorem*{theorem*}{Theorem}
\newtheorem*{corollary*}{Corollary}
\newtheorem*{lemma*}{Lemma}
\newtheorem*{remark*}{Remark}
\newtheorem*{example*}{Example}

\newcommand{\R}{\mathbb{R}}
\newcommand{\C}{\mathbb{C}}

\newcommand{\tr}{\mbox{Tr}}

\begin{document}

\title{Master Equation and Perturbative Chern-Simons theory}

\author{Vito Iacovino}

\address{}

\email{iacovino@mpim-bonn.mpg.de}

\date{version: \today}


\begin{abstract}
We extend the Chern-Simons perturbative invariant of Axelrod and Singer \cite{AS} to non-acyclic connections. We construct a solution of the quantum master equation on the space of functions on the cohomology of the connection. We prove that this solution is well defined up to master homotopy. 
We discuss also invariants of links.
\end{abstract}

\maketitle

\section{introduction}
Let $M$ be a compact oriented three manifold. Consider a flat connection on a principal bundle over $M$ with compact structural group. Let $\mathfrak{g}$ be the related Lie algebra bundle. 
 
If the cohomology $H^*(M,\mathfrak{g})$ of the flat connection is trivial, Axelrod and Singer (\cite{AS}) and Kontsevich (\cite{Ko}) proved that the perturbative expansion of the Chern-Simons theory leads to topological invariants of the manifold $M$. 

Non acyclic connections have been recently considered by Costello (\cite{Co}). 
The perturbative expansion of the partition function should lead to a function on the cohomology of the connection $H^*(M,\mathfrak{g})$ that solves the quantum master equation and is well defined up to master homotopy. The coefficients of the solution can be considered as a quantum generalization of the Massey products. In (\cite{Co}), Costello was able to construct the solution up to the constant term.
His solution was found as application of the general theory for the quantization and renormalization of gauge theories developed in \cite{Co} and using an abstract local to global argument.

In this paper we construct the full solution of the master equation. The solution is written in terms of a perturbative expansion in such a way that it is not necessary to renormalize the theory.
We prove that up to master homotopy only the constant term of the perturbative expansion depends on the metric. The dependence on the metric can be canceled by subtracting an appropriate multiple of the gravitational Chern-Simons invariant. As in \cite{AS} this involves a choice of frame of $TM$.


The solution of the master equation is written, analogously to \cite{AS}, in terms of an expansion of Feynman graphs. In this case the trivalent graphs are allowed to have external edges. To any graph is associated a polynomial on $H^*(M,\mathfrak{g})$ integrating a differential form on the space of the space of configurations of its vertices.

The technical part of \cite{AS} was devoted to the study of the physical propagator and the related analysis of the finiteness of the theory. Axelrod and Singer were able to prove that the kernel of the physical propagator
defines a smooth differential form on $C_2(M)$ (the blowup of $M^2$ over the diagonal) providing a geometric description of the singularity of the kernel along the diagonal.  
We avoid these technical issues using a geometric approach that generalize the approach of Kontsevich (see also \cite{BC}). Instead of studying the physical propagator we define the propagator directly as a differential form on $C_2(M)$ which satisfies some conditions that are defined in terms of some geometric data. The data include a metric on $M$, a connection compatible with the metric, and a vector subspace of $\Omega^*(M,\mathfrak{g})$ representing $H^*(M,\mathfrak{g})$. We prove that two different choices of such data lead to solutions of the Master equation that are Master homotopic.


We extend the analysis to the study of link invariants. To a link in $M$ is associated an observable of the BV-formalism. We prove that the observables associated to equivalent links are homotopic up to an anomaly term. For framed links, the anomaly can be removed modifying the observable. 

We also study how the observable changes if the link self-intersects. In a neighborhood of the intersection points of the link it is necessary to consider a new kind of compactification of the configuration of the points. It turns out that the jump of the observable is tied to the Chas-Sullivan string product of the family of the links.

During the preparation of this note, we have become aware of independent work by Cattaneo and Mnev \cite{CM} on the same topic.

{\bf Acknowledgements. } We are grateful to K. Costello for helpful discussions and to C. Rossi for his comments.

\section{Quantum Master Equation}

In this section we recall some basic definition related to the (finite dimensional) Batalin-Vilkovski formalism. For more details see (\cite{Co}). 

Fix a super vector space $H$ with an odd symplectic form. Denote by $\mathcal{O}(H)$ the algebra of polynomial functions on $H$.

Let $x_i$, $y_i$ be Darboux coordinates for $H$ with $x_i$ even and $y_i$ odd. Let $\Delta$ be the order two differential operator on $H$ defined by  
$$ \Delta = \partial_{x_i} \partial_{y_i} . $$
The operator $\Delta $ is independent of the choice of basis of $H$.

The bracket on the algebra $\mathcal{O}(H)$ is defined by
$$ \{  f,g  \} = \Delta (fg) - \Delta (f ) g - (-1)^{|f|} f \Delta (g). $$


Denote by $\mathcal{O}(H)[\hbar]$ the polynomial functions with coefficients in the formal parameter $\hbar$. An even element $S \in \mathcal{O}(H)[\hbar]$ satisfies the quantum master equation if 
$$  \Delta e^{S/\hbar}=0 . $$
This equation can be rewritten as
\begin{equation} \label{masterequation}
\frac{1}{2} \{ S, S \} + \hbar \Delta S =0 .
\end{equation}

We will need to consider also the one parameter family version of the above construction.
Consider the space $\Omega^*([0,1])  \otimes \mathcal{O}(H) [\hbar]$. Extend the operator $\Delta$ to this space acting trivially on $\Omega^*([0,1]) $.
A master homotopy is an even element $ \tilde S \in  \Omega^*([0,1]) \otimes  \mathcal{O}(H) [\hbar] $ such that 
\begin{equation} \label{masterhomotopy}
d \tilde S + \frac{1}{2} \{ \tilde S,\tilde S \} + \hbar \Delta \tilde S =0 . 
\end{equation}
Write $\tilde S$ as $ \tilde S = A(t) + B(t) dt$ where $A(t)$ and $B(t)$ are elements of $\mathcal{O}(H) [\hbar]$. Equation (\ref{masterhomotopy}) becomes
$$ \frac{1}{2} \{ A(t), A(t) \} + \hbar \Delta A(t) =0     $$
$$  \dot A(t) +  \{ B(t), A(t) \} + \hbar \Delta B(t) =0   .  $$

In the case we are interested in $H$ is the cohomology of a fixed flat connection 
$$H = H^*(M, \mathfrak{g})[1] .$$ 
The odd symplectic form is induced by the pairing
$$\langle \alpha \otimes X,  \alpha' \otimes X' \rangle = (-1)^{|\alpha|} \int_M \alpha \wedge \alpha'  \langle  X,  X' \rangle_{\mathfrak{g}} $$

\section{Effective action}

Let $C_n(M)$ denote the configuration space of $n$ points in $M$. The boundary of $C_2(M)$ is isomorphic to the $2$-sphere bundle $S(TM)$ of $TM$. 
We will often consider the differential forms on $M \times M$ as subspace of the differential forms on $C_2(M)$.
Also, the differential forms on $C_2(M)$ can be considered as differential forms on $M \times M$ with particular type of singularity along the diagonal.

In this section we construct a version of the propagators of \cite{AS} and \cite{Co} as a differential form on $C_2(M)$. 
We need to fix the following data:
\begin{itemize}
\item a metric on $M$

\item a connection on $TM$ compatible with the metric 

\item a vector space $\Psi \subset \Omega^*(M,\mathfrak{g})$ of closed forms such that the natural projection
$$ \Psi \rightarrow   H^*(M, \mathfrak{g})[1]$$
is an isomorphism.
\end{itemize}

Let $x_i$, $y_i$ be Darboux coordinates for $H^*(M, \mathfrak{g})[1]$ and let $\alpha_i $, $\beta_i$ be the associated basis of $\Psi$.
Define $ \psi \in \mathcal{O}(H) \otimes  \Omega^*(M,\mathfrak{g}) $
using 
\begin{equation} \label{psi}
\psi = \sum_i x_i \alpha_i + y_i  \beta_i .
\end{equation}


Define $K \in \Omega^3(M^2,\pi_1^*\mathfrak{g} \otimes \pi_2^*\mathfrak{g} ) $ as 
\begin{equation} \label{kappa}
K =  \sum_i \alpha_i \otimes \beta_i + \beta_i \otimes \alpha_i .
\end{equation}

The differential forms $\psi$ and $K$ do not depend on the Darboux coordinates we used.

\subsection{Propagator}
Fix a local orthogonal frame of $TM$. The bundle $S(TM)$ is a trivial bundle with fiber $ S^2 $. 
Denote by $\theta_i$ the $1$-form components of the connection in this local system.  
Define the differential form 
\begin{equation} \label{singularity}
 \eta = \frac{\omega + d(\theta^i x_i)}{ 4 \pi}
\end{equation}
where $\omega$ is the standard volume form of $S^2$ and $x_i$ are the restriction to $S^2$ of the standard coordinates of $\R^3$. The form (\ref{singularity}) is independent of the choice of the local frame of $TM$. Therefore the differential form $ \eta$  is defined globally on $ \Omega^2(S(TM))$.



Denote by $\pi_{\partial} : \partial C_2(M) \rightarrow M $ the natural projection. Let $I_{\mathfrak{g}} \in \pi_1^*(\mathfrak{g}) \otimes \pi_2^*(\mathfrak{g})$ be the tensor dual of the pairing on $\mathfrak{g}$. Let $r :\Omega^2(C_2(M), \pi_1^*(\mathfrak{g}) \otimes \pi_2^*(\mathfrak{g})) \rightarrow \Omega^2(C_2(M), \pi_1^*(\mathfrak{g}) \otimes \pi_2^*(\mathfrak{g})) $ be the map induced by the reflection map $(x,y) \rightarrow (y,x)$ on $M^2$.
\begin{lemma} \label{propagator}
There exists a differential form $P \in \Omega^2(C_2(M), \pi_1^*(\mathfrak{g}) \otimes \pi_2^*(\mathfrak{g}))$ such that
\begin{equation} \label{singularity2}
i_{\partial}^*  P =  \eta \otimes I_{\mathfrak{g}} + \pi_{\partial}^* (\phi)
\end{equation}
for some $\phi \in  \Omega^2( M , \pi_1^*(\mathfrak{g}) \otimes \pi_2^*(\mathfrak{g}))$,
\begin{equation} \label{differential}
d P = K
\end{equation}
\begin{equation} \label{parity}
r^*  P = - P
\end{equation}  
and
\begin{equation} \label{contraction}
\langle P, \alpha_1 \otimes \alpha_2 \rangle =0
\end{equation}
for each $ \alpha_1, \alpha_2 \in \Psi $. 
 

Moreover $P$ is unique up addiction of the differential of a form in $\Omega^1(C_2(M), \pi_1^*(\mathfrak{g}) \otimes \pi_2^*(\mathfrak{g}))$ with pull-back on $\partial C_2(M)$ in $\pi_{\partial}^* (\Omega^1( M , \pi_1^*(\mathfrak{g}) \otimes \pi_2^*(\mathfrak{g}))$.
\end{lemma}
\begin{proof}
Let $U$ be a small tubular neighborhood of the diagonal $\Delta$ of $M \times M$. There is natural induced map $\pi_U : U \rightarrow S(TM) $. Let $\rho$ be a cutoff function equal to one in a neighborhood of $\partial C_2(M)$ and zero outside a compact subset of $U$. 
If $U$ is small enough we can use the parallel transport along the radii in order to identify the fiber of the bundle $\mathfrak{g}$. Using this trivialization we can extend $I_{\mathfrak{g}}$ to a parallel section $I_{\mathfrak{g}} \in \Omega^0(U, \pi_1^*(\mathfrak{g}) \otimes \pi_2^*(\mathfrak{g}))$. 
Using this identification we can define preliminarily $ P$ as 
$$ P = \rho (\pi_U^* \eta) \otimes I_{\mathfrak{g}}.$$ 
Equation (\ref{singularity2}) holds for $\phi = 0$:
\begin{equation} \label{singularity2pr}
i_{\partial}^*  P =  \eta \otimes I_{\mathfrak{g}} + \pi_{\partial}^* (\phi)
\end{equation} 

In the following we will omit in the notation the coefficient bundle. All the differential forms and cohomology groups have coefficients in the bundle $\pi_1^*(\mathfrak{g}) \otimes \pi_2^*(\mathfrak{g})$

The differential form $ P$ is closed in a neighborhood of $S(TM)$, therefore we can consider $d P$ as a closed form on $\Omega^2(M \times M)$. For any closed differential form $\tau \in \Omega^3(M \times M)$, integrating by parts we have
$$ \int_{M^2} (d P) \wedge \tau =  \int_{C_2(M)} (d P) \wedge \tau = \int_{S(TM)} P \wedge i^*_{\Delta} \tau = \int_{\Delta} \tau$$
where in the last equality we have applied (\ref{singularity2pr}).
It follows that $dP$ and $K$ are in the same cohomology class in $\Omega^3(M \times M)$. Therefore there exists a differential form $\alpha \in \Omega^2(M \times M)$ such that 
$$ K = dP + d \alpha . $$

Replace $P$ with $P + \alpha$. Equation (\ref{differential}) holds. Now equation (\ref{singularity2}) holds with $\phi = i^*_{\Delta} \alpha$. In the same way we can add to $P$ a closed form of $\Omega^2(M^2)$ such that also (\ref{contraction}) holds. 

$P$ will also satisfy (\ref{parity}) if we choose the cut off function $\rho$ such that $T^* \rho = \rho$ and the differential forms that we add to $P$ are antisymmetric.

Now suppose that $P'$ is another element of $\Omega^2(C_2(M))$ such that (\ref{singularity2}), (\ref{differential}), (\ref{contraction}) and (\ref{parity}) hold. Let $\phi'$ be the corresponding form in (\ref{singularity2}). Consider the following commutative diagram
$$
\xymatrix{ \ar[r] & H^2(C_2(M),S) \ar[r] & H^2(C_2(M))  \ar[r] & H^2(S)    \ar[r] & H^3(C_2(M),S) \ar[r] & \\
           \ar[r] & H^2(M \times M, \Delta) \ar[r] \ar[u]^{\sim} & H^2(M \times M) \ar[r] \ar[u] & H^2(\Delta) \ar[r] \ar[u] &  H^3(M \times M, \Delta) \ar[u]^{\sim} \ar[r] & }
$$
where the rows are exact sequences.
$P'-P$ defines an element of $H^2(C_2(M))$ and $\phi' - \phi$ defines an element of $H^2(\Delta)$. These two elements have the same image on $H^2(S)$. From the commutativity of the diagram it follows that $\phi' - \phi$ is mapped to zero on $H^3(M \times M , \Delta)$ and therefore there exists $\alpha \in \Omega^2(M \times M) $ such that 
$$ i^*_{\Delta} \alpha  = \phi' -\phi .$$ 
The differential form $P' - P - \alpha $ defines an element of $H^2(C_2(M),S) $. Since $H^2(C_2(M),S)  \cong H^2(M \times M, \Delta) $ there exist $\beta \in \Omega^2( M \times M)$ and $\varphi \in \Omega^2(C_2(M))$ such that
$$ P' - P - \alpha = \beta + d \varphi$$
with $i^*_{S} \varphi =0 $.
Property (\ref{contraction}) applied to $P'-P$ implies that $\alpha + \beta $ is cohomologicaly trivial on $\Omega^2(M \times M)$.

\end{proof}

\subsection{Effective Action}

Let $\gamma$ be a trivalent graph that can have external edges. We allow edges starting and ending at the same vertex. Denote by $V(\gamma)$ and $E(\gamma)$ the sets of vertices and edges of $\gamma$. 

For $v \in V(\gamma)$ let $\pi_v : C_{V(\gamma)}(M) \rightarrow M $ be the projection on the point $v$ and define
$$ \mathfrak{g}_v = \pi_v^* (\mathfrak{g}) . $$
For $e \in E(\gamma) $ let $\pi_e$ be the projection on the vertices attached to $e$. We have $\pi_e : C_{V(\gamma)} \rightarrow C_2(M) $ if $e$ is an internal edge connecting two different vertices and $\pi_e : C_{V(\gamma)} \rightarrow M $ if $e$ is external edge or an edge starting and ending on the same vertex. 

As in \cite{AS}, in order to make the signs simpler it is useful to introduce the super-propagator $P_s$ as the image of $P$ by the inclusion
$$ \pi_1^*(\mathfrak{g}) \otimes \pi_2^*(\mathfrak{g})   \rightarrow \bigwedge(\pi_1^*(\mathfrak{g}) \oplus \pi_2^*(\mathfrak{g})) . $$
Property $(\ref{parity})$ for $P$ implies 
$$ r^* (P_s) =P_s .$$

Define the bundle $\mathfrak{g}_{V(\gamma)}$ over $C_{V(\gamma)}(M)$ by
$$ \mathfrak{g}_{V(\gamma)} = \bigwedge (\bigoplus_{v \in V(\gamma)} \mathfrak{g}_v) .$$

To the graph $\gamma$ is associated the differential form $\omega_{\gamma} \in  \Omega^*(C_{V(\gamma)}(M)) \otimes  \mathfrak{g}_{V(\gamma)}$
defined by
\begin{equation} \label{graph}
\omega_{\gamma}  = \bigwedge_{e \in E^{in}(\gamma)} \pi^*_e  P_s.
\end{equation}
In the formula (\ref{graph}), if $e$ is an edge starting and ending at the same vertex we define $\pi^*_e  P_s = \pi^*_e  \phi_s $.

For any vertex $v \in V(\gamma)$ define
$$ \text{Tr}_v :   \mathfrak{g}_{V(\gamma)} \rightarrow   \mathfrak{g}_{V(\gamma)}$$
as follows.
Let $X_i \in \mathfrak{g}_v $ for $1 \leq i \leq k$ and $\tilde{X} \in \mathfrak{g}_{V(\gamma)} $ without components in $\mathfrak{g}_v$.
Then
$$  \tr_v(X_1 \wedge X_2 \wedge .... \wedge X_k \wedge \tilde{X})= 0$$
if $k \neq 3$ and 
$$  \tr_v(X_1 \wedge X_2 \wedge X_3 \wedge \tilde{X})= \langle  X_1,[X_2,X_3]  \rangle \tilde{X}$$
if $k=3$. 

The composition of the $\{ \text{Tr}_v \}_{v \in V(\gamma)}$ defines 
$$  \text{Tr}_{V(\gamma)} = \otimes_{v \in V(\gamma)} \text{Tr}_v : \mathfrak{g}_{V(\gamma)} \rightarrow \C. $$

The effective action $S$ is defined by
\begin{equation} \label{action}
S =   \sum_{\gamma} \frac{1}{\text{Aut}(\gamma)} {\hbar}^{l(\gamma)}  \int_{C_{V(\gamma)}(M)} \text{Tr}_{V(\gamma)} (\omega_{\gamma} \wedge \bigwedge_{e \in E^{ex}(\gamma)}  \pi_e^*( \psi)) .
\end{equation}
where $l(\gamma)$ is the number of loops of the graph $\gamma$. Observe that in order to fix the sign of $\tr_{V(\gamma)}$ and the orientation of $C_{V(\gamma)}(M)$ it is necessary to order the vertices of $\gamma$ up to even perturbations. Since these two signs cancel, definition (\ref{action}) works without ambiguity. 

\begin{theorem} 
$S$ satisfies the master equation (\ref{masterequation}). 

If $S_0$ and $S_1$ are solutions associated to two different sets of data there exists $\tilde{S} \in \Omega^*([0,1]) \otimes \mathcal{O}(H^*(M)) $ such that $\tilde{S}|_0=S_0 $, $\tilde{S}|_1=S_1 $ and
\begin{equation} \label{homotopyformula}
d \tilde S + \frac{1}{2} \{ \tilde S , \tilde S \} + \hbar \Delta \tilde S =  \beta(\hbar) \int_M p(\tilde \theta).
\end{equation}
In equation (\ref{homotopyformula}), $ \beta(\hbar)$ is a formal series in $\hbar$ which is independent of $M$ and $p(\tilde{\theta})$ is the Pontryagin class of the connection $\tilde{\theta}$ on $T(M \times I)$.
\end{theorem}

Formula (\ref{homotopyformula}) is proved in Proposition \ref{homotopy}. The first part of the theorem follows from (\ref{homotopyformula}) applied to a constant family of data. 

Formula (\ref{homotopyformula}) is the master homotopy equation up to the anomaly of \cite{AS}. In order to find an actual homotopy we need to fix an orthonormal frame of $TM$ in order to modify the effective action.
Denote by $CS (\theta)$ the gravitational Chern-Simons invariant of the connection associated this frame (cf. \cite{AS}, \cite{BC}). 
This is defined by
$$\text{CS}(\theta) = \int_M ( \theta^i d \theta_i -\frac{1}{3} \epsilon_{ijk}  \theta^i  \theta^j  \theta^k ) $$ 
where $ \theta_i$ are the components of the connection in the frame.

\begin{corollary}
For two different sets of data $S - \beta(\hbar) CS(\theta) $ are master homotopic. 
\end{corollary}
\begin{proof}
Given a one parameter family of connections and an orthonormal frame of $T(M \times [0,1])$ we can define the extended gravitational Chern-Simons functional as
$$\text{CS}(\tilde \theta) = \int_M ( \tilde \theta^i d \tilde \theta_i -\frac{1}{3} \epsilon_{ijk} \tilde \theta^i \tilde \theta^j \tilde \theta^k ) $$ 
where $\tilde \theta_i$ are the components of the connection in the frame.
As in \cite{AS} we have 
$$ d \text{CS}(\tilde \theta) = \int_M p(\tilde \theta) .$$
The corollory follows from Formula (\ref{homotopyformula}).
\end{proof}

\section{Invariance}
\subsection{Extended propagator}

In this section we extend the construction of the propagator to a family of data. We have a smooth family of data parametrized by the interval $I= [0,1]$, that is 

\begin{itemize}
\item a family of metrics 
\item a family of compatible connections 
\item a family of vector spaces $\Psi_t \subset \Omega^*(M)$.
\end{itemize}

The family metrics and connection define a metric and a compatible connection on $M \times I$ respectively.

For $\alpha \in H^*(M,\mathfrak{g} )$, denote by $\alpha_0(t)$ the element in $\Psi_t$ representing the class $\alpha$. 
There exists $\alpha_1(t)  \in \Omega^*(M,\mathfrak{g})$ such that 
$$\frac{d}{dt} \alpha_0(t) = - d \alpha_1(t) $$ 
and 
$$\langle \alpha_1(t) , \Psi_t \rangle =0 .$$
Denote $\tilde \alpha = \alpha_0 (t) + \alpha_1 (t)dt$. This defines a linear map $ \tilde \Psi : H^*(M,\mathfrak{g} ) \rightarrow \Omega^*(M \times I,\mathfrak{g})$.
 

Let $S(TM \times I)$ be the unit sphere bundle of $TM \times I \rightarrow M \times I$. 
In analogy with formula (\ref{singularity}), define the differential form $ \tilde \eta \in \Omega^2(S(TM \times I))$ locally as
$$ \tilde{\eta} = \frac{\omega + d(\tilde{\theta}^i x_i)}{ 4 \pi} $$
using a local orthonormal frame of $T(M \times I)$. 

Let $( \alpha_i, \beta_i)$ be a Darboux basis of $H^*(M, \mathfrak{g})[1]$, and let $(\tilde \alpha_i, \tilde \beta_i)$ be the associated elements in $\Omega^*(M \times I,\mathfrak{g})$ through $\tilde \Psi$.
Define
\begin{equation} \label{fpsi}
\tilde \psi = \sum_i x_i \tilde \alpha_i +  y_i \tilde \beta_i .
\end{equation}
and
\begin{equation} \label{fkappa}
\tilde K =  \sum_i \tilde \alpha_i \otimes \tilde \beta_i + \tilde \beta_i \otimes \tilde \alpha_i .
\end{equation}

\begin{lemma} \label{fpropagator}
There exists a differential form 
$$\tilde{P} = P_0(t) + P_1(t) dt \in \Omega^2(C_2(M) \times I, \pi_1^*(\mathfrak{g}) \otimes \pi_2^*(\mathfrak{g}))$$ 
such that
\begin{equation} \label{fsingularity}
i_{\partial}^* \tilde  P = \tilde \eta \otimes I_{\mathfrak{g}} + \pi_{\partial}^* (\tilde \phi)
\end{equation}
for some $\tilde \phi \in \Omega^2(M \times I, \pi_1^*(\mathfrak{g}) \otimes \pi_2^*(\mathfrak{g}))$ 
\begin{equation} \label{fdifferential}
d \tilde P = \tilde K 
\end{equation}
\begin{equation} \label{fcontraction}
\langle P_0(t), \alpha_0(t) \otimes \beta_0(t) \rangle =0.
\end{equation}
for any $\tilde \alpha, \tilde \beta \in \tilde \Psi $ and $t \in I$, and $T^* \tilde P = - \tilde P$ hold.

Moreover $\tilde P$ is unique up to the addition of the differential of a form in $\Omega^1(C_2(M) \times I, \pi_1^*(\mathfrak{g}) \otimes \pi_2^*(\mathfrak{g}))$ with pull-back on $\partial C_2(M) \times I$ in
$\pi_{\partial}^* (\Omega^1( M \times I, \pi_1^*(\mathfrak{g}) \otimes \pi_2^*(\mathfrak{g})))$.
\end{lemma}

\begin{proof}

Using the same argument of Lemma \ref{propagator} we can construct a differential form $\tilde P \in \Omega^2(C_2(M) \times I, \pi_1^*(\mathfrak{g}) \otimes \pi_2^*(\mathfrak{g}))$ such that (\ref{fsingularity}) and (\ref{fdifferential}) holds for some $\tilde \phi \in  \Omega^2( M \times I , \pi_1^*(\mathfrak{g}) \otimes \pi_2^*(\mathfrak{g}))$. The condition (\ref{fcontraction}) can be imposed using the Lemma \ref{dcontraction}.

\end{proof}
\begin{lemma} \label{dcontraction}

With the same notation of Lemma \ref{fpropagator} the following holds
$$ \frac{d}{dt}\langle P_0(t), \alpha_0(t) \otimes \beta_0(t) \rangle =0. $$
\end{lemma}

\begin{proof}
Write $\tilde{K}= K_0 + K_1 dt$. Equation (\ref{fdifferential}) can is equivalent to $d P_0 = K_0$ and $ \dot{P}_0 + d P_1 = K_1$. We have  
$$ \frac{d}{dt}\langle P_0(t), \alpha_0(t) \otimes \beta_0(t) \rangle = \langle \dot{P}_0(t), \alpha_0(t) \otimes \beta_0(t) \rangle + \langle P_0(t), \dot{\alpha}_0(t) \otimes \beta_0(t) \rangle + \langle P_0(t), \alpha_0(t) \otimes \dot{\beta}_0(t) \rangle .$$
We now prove that each term in the left hand side is zero.

Since $d \alpha_0 = -d \alpha_1$, integrating by parts we have
$$ \langle P_0(t), \dot{\alpha}_0(t) \otimes \beta_0(t) \rangle = - \langle P_0(t), d \alpha_1(t) \otimes \beta_0(t) \rangle = \langle d P_0(t), \alpha_1(t) \otimes \beta_0(t) \rangle + \langle \alpha_1(t) , \beta_0(t) \rangle =0.$$
In the same way we can prove that $\langle P_0(t), \alpha_0(t) \otimes \dot{\beta}_0(t) \rangle =0$

Since $ \dot{P}_0 = K_1 - d P_1 $ in order to prove that $ \langle \dot{P}_0(t), \alpha_0(t) \otimes \beta_0(t) \rangle =0 $ it is enough to prove
$$ \langle K_1 , \alpha_0(t) \otimes \beta_0(t) \rangle =0 $$
$$ \langle  d P_1(t), \alpha_0(t) \otimes \beta_0(t) \rangle =0 .$$
The first is immediate using the definition of $K_1$. The second follows again by integration by part, where now the boundary term is zero because the push-forward on the diagonal $\Delta$ of $\tilde{\eta}$ has no component in $d t$.

\end{proof}


\subsection{Master Homotopy}

Using the extended propagator $\tilde P$ we can extend formula (\ref{graph}) to a one parameter family of data. For any graph $\gamma$ let
\begin{equation} \label{fgraph}
\tilde \omega_{\gamma}  = \bigwedge_{e \in E^{in}(\gamma)} \pi^*_e \tilde  P_s \in \Omega^*(C_{V(\gamma)}(M) \times I,   \mathfrak{g}_{V(\gamma)} ).
\end{equation}

Define the extended effective action $\tilde S \in \Omega^*(I) \otimes \mathcal{O}(H)$ using
\begin{equation} \label{faction}
\tilde S =   \sum_{\gamma} \frac{1}{\text{Aut}(\gamma)} {\hbar}^{l(\gamma)}  \int_{C_{V(\gamma)}(M)} \text{Tr}_{V(\gamma)} (\tilde \omega_{\gamma} \wedge \bigwedge_{e \in E^{ex}(\gamma)}  \pi_e^*( \tilde \psi)) .
\end{equation}
where now we consider the integrals as push forward on the interval $I$.


\begin{lemma} \label{face}
Let $\delta$ be a trivalent graph with $k$ external edges. Define
$$\mathcal{S}_{\delta} = \sqcup_{(p,t) \in M \times I} C_{V(\delta)}(T_pM)/ \sim $$ 
where $\sim$ is the equivalence by homotheties and translations.
$\mathcal{S}_{\delta}$ can be identified with the subset of $C_{V(\delta)}(M) \times I$ where all the vertices are collapsed on a point and there is natural projection $\pi_{\delta} :\mathcal{S}_{\delta} \rightarrow M \times I$ that is a fiber bundle. 


Define
$$c_{\delta} = (\pi_{\delta})_*  \tilde \omega_{\delta} \in  \Omega^*(M \times I) \otimes  \mathfrak{g}_{V(\delta)} .$$

Then, if $\delta$ has more than two vertices $c_{\delta}$ is zero unless $k=0$. In this case $c_{\delta} $ is a multiple of the Pontryagin class $p(\tilde \theta)$. 

If $\delta$ has two vertices $1$ and $2$ we have the following cases. Let $n_i$ and $l_i$ be the number of external edges and closed edges attached to $i$. Let $m$ be the number of edges connecting $1$ and $2$.

\begin{itemize}

\item $n_1=n_2=l_1=l_2=0 $ and $m=3$. Then $c_{\delta} = p(\tilde \theta) + 3 \tilde \phi_{12} \wedge \tilde \phi_{12}  \wedge I_\mathfrak{g}$.

\item $n_1=n_2=0$, $l_1=l_2=1$ and $m=1$. Then $c_{\delta} = 2 \tilde \phi_1 \wedge \tilde \phi_{2}  \wedge I_\mathfrak{g}$.  

\item $n_1=n_2=1$, $l_1=l_2=0$ and $m=2$. Then $c_{\delta} = 2 \tilde \phi_{12} \wedge I_\mathfrak{g}$.

\item  $n_1=2$, $n_2=0$, $l_1=0$, $l_2=1$ and $m=1$. Then $c_{\delta} = \tilde \phi_{2} \wedge I_\mathfrak{g}$.

\item $n_1=n_2=2$, $l_1=l_2=0$ and $m=1$. Then $c_{\delta} =  I_\mathfrak{g}$.

\end{itemize}
Where we consider $\tilde \phi_{i}$ with coefficients in the bundle $\wedge^2(\mathfrak{g}_i)$.




\end{lemma}
\begin{proof}

We can write $\tilde \omega_{\delta}$ as  
$$ \tilde \omega_{\delta} = \sum_S  \bigwedge_{e \in E^{in}(\delta) \setminus S} \tilde \eta \wedge \bigwedge_{e \in S} \pi_e^* (\tilde \phi) $$
where the sum is done on all the subsets $S$ of $E^{in}(\delta)$.
Since the differential forms $\pi_e^*(\tilde \phi)$ descend to the differential forms on the base $M \times I$
we can write $c_{\delta}$ as 
$$ \tilde c_{\delta} = \sum_S c_{\delta}^S \wedge \bigwedge_{e \in S} \pi_e^* (\tilde \phi). $$

Consider first the coefficient $c_{\delta}^0$ of the contribution of the empty set $S= \emptyset$. $ c_{\delta}^0$ is a differential form of degree $4-k$ with coefficients in the flat bundle $\mathfrak{g}_{V(\delta)}$.
$c_{\delta}^0$ has to be an invariant polynomial in $\theta$ and $d \theta$. Therefore $k=0$ or $4$. 

If $k=0$, $ c_{\delta}^0$ is a $4$ differential form on $M \times I$ that is proportional to the Pontryagin class.

If $k=4$, $ c_{\delta}^0$ is a zero differential form and therefore the push forward selects the part of degree zero in $\theta$. Hence we can apply the vanishing theorem of Kontsevich (see \cite{Ko}, \cite{BC}). This implies that $\delta$ has only two vertices connected exactly by an internal edge.

Consider now the term $c_{\delta}^S$ for $S \neq \emptyset$.
Consider the graph $\delta'$ obtained by "cutting" the edges in $S$, that is replace all the edges of $S$ with two external edges. The previous argument applied to $\delta'$ implies that if $c_{\delta}^S \neq 0$ then $\delta'$ is the graph composed by two vertices connected by an internal edge and having four external edges. The result follows.

\end{proof}

\begin{proposition} \label{homotopy}
$\tilde{S}$ is a solution of the homotopy equation with anomaly (\ref{homotopyformula}).
\end{proposition}
\begin{proof}
The proof is based on the application of Stokes theorem to each term in the sum (\ref{faction}). For any fixed graph $\gamma$ this gives the identity
\begin{equation} \label{stokes}
d  \int_{C_{V(\gamma)}(M)} +  \int_{C_{V(\gamma)}(M)} d = \int_{\partial C_{V(\gamma)}(M)} .
\end{equation}
The first term of (\ref{stokes}) generates $d \tilde P$. 
For the second term observe that
\begin{equation} \label{break}
 d  \tilde \omega_{\gamma}  =  \sum_{e \in E^{in}(\gamma)}   \pi^*_e ( \tilde K) \wedge \bigwedge_{e' \in E^{in}(\gamma) \setminus e} \pi^*_{e'}  \tilde P.
\end{equation}
Therefore the second term breaks into two contributions. The edges $e$ disconnecting the graphs $\gamma$ generate $\frac{1}{2} \{\tilde S ,\tilde S \} $. The edges $e$ not disconnecting the graph $\gamma$ generate $\hbar \Delta \tilde S$.

We are left to prove that the boundary term of (\ref{stokes}) yields the right side of (\ref{homotopyformula}). 
The boundary of $C_{V(\gamma)}(M) \times I$ is union of faces, each of which corresponds to a collapse of a subset of vertices of $\gamma$ to a point.

Given a subset of $V(\gamma)$ there exists a unique trivalent subgraph of $\gamma$ with these as vertices (the edges are given by all the edges of $\gamma$ starting from the vertices). 

Let $\delta$ be a trivalent subgraph of $\gamma$. Observe that the external edges of $\delta$ correspond to the edges of $\gamma$ attached to exactly a vertex of $\delta$. To $\delta$ corresponds a boundary face of $C_{V(\gamma)}(M)$ in the following way. 

Let $ \pi_{\delta} :\mathcal{S_{\delta}} \rightarrow M \times I$ be the bundle as in Lemma \ref{face}. Let $\gamma'$ be the graph obtained from $\gamma$ contracting $\delta$ to a vertex. Let $p_{\delta}:C_{V(\gamma')}(M) \times I \rightarrow M \times I$ be the map defined by the point which is mapped the vertex $\delta$. The boundary face associated to $\delta$ is the bundle
\begin{equation} \label{bundleface}
\pi_{\delta} : p_{\delta}^*{\mathcal{S_{\delta}}} \rightarrow  C_{V(\gamma')}(M) \times I .
\end{equation}

The restriction of $\tilde \omega_{\gamma}$ to this boundary face is given by $\pi_{\delta}^* \tilde \omega_{\gamma'} \wedge p_{\delta}^*(\tilde \omega_{\delta})$.
Its push forward by (\ref{bundleface}) is given by $ \tilde \omega_{\gamma'} \wedge p_{\delta}^* (c_{\delta})$ where $c_{\delta}$ is defined in Lemma \ref{face}. From Lemma \ref{face} follows that it is zero
unless $\delta= \gamma$ or $\delta$ has two vertices. 
The contribution of boundary faces associated to graphs $\delta$ with two vertices joined by exactly an internal edge cancel because of the Jacobi identity.

\end{proof}



\section{Link Invariants}

A link on $M$ is a finite set of closed curves on $M$. A link can be represented by an embedding of a one dimensional manifold $N$ into $M$:
$$ \alpha_0 : N \rightarrow M . $$   
In this section we associate to a link $\alpha_0$ on $M$ an invariant given by an observable $ \mathcal{O}_{\alpha_0} \in  \mathcal{O}(H^*(M)) [\hbar] $ of BV formalism. An element $ \mathcal{O}_{\alpha} \in  \mathcal{O}(H^*(M)) [\hbar] $ is an observable if
\begin{equation} \label{observable}
\hbar \Delta \mathcal{O_{\alpha}} + \{ S, \mathcal{O_{\alpha}}  \} =0 .
\end{equation}
Observe that the map 
$$ \mathcal{O} \rightarrow \hbar \Delta \mathcal{O} + \{ S, \mathcal{O}  \}  $$
is a linear map with square zero. Its homology is the homology of the observables.  

The observable $ \mathcal{O}_{\alpha_0}$ will depend not only on the choice of flat connection of the $\mathfrak{g}$-bundle, but also on the choice of a representation
$$ \rho :  \mathfrak{g} \rightarrow \mathfrak{gl}(n, \C) .$$
In the following we implicitly identify $\mathfrak{g}$ with a sub-Lie-algebra of $\mathfrak{gl}(n, \C)$ using $\rho$.  

\subsection{Configuration of points}

Consider a one parameter family of maps between $N$ and $M$
\begin{equation} \label{map0}
\alpha_s : N \rightarrow M
\end{equation}
parametrized by points $s \in [0,1]$.
The family (\ref{map0}) is defined by a smooth map
\begin{equation} \label{map}
\alpha : [0,1] \times N \rightarrow M .
\end{equation}

Let $m,n$ be positive integers with $m \geq n$. We want to construct the manifold $C_{m,n}(\alpha)$ of the configuration space of $m$ points on $M$ with $n$ points living on $N$.  

Suppose first that the map (\ref{map0}) is an embedding for each $s \in [0,1]$. In this case there is an induced map of configuration space of points 
$$  [0,1] \times C_n(N) \rightarrow C_n(M) . $$
The fibered product of this map with the obvious projection $C_n(M) \rightarrow C_m(M)$ gives
$$ C_{m,n}(\alpha)=([0,1] \times C_n(N)) \times_{C_n(M)} C_m(M) .$$

We want also to consider the possibility that the link can self-intersect.
Therefore the maps (\ref{map0}) are immersions but they can fail to be injective. In this case we need to modify the construction of the configuration of points in the following way.
Let $V$ be an open sub-interval of $[0,1]$ and let $U_1, U_2$ be disjoint open subsets of $N$. Assume that for each $s \in V$, the map $\alpha_s$ is injective on each $U_i$.
Then there is an induced map of configuration of points
$$V \times C_{n_1}(U_1) \times C_{n_2}(U_2) \rightarrow C_{n_1}(M) \times C_{n_2}(M)$$
where $n_1 + n_2 = n$. As before we consider the fibered product of this map with the natural map $C_m(M) \rightarrow \Pi_i C_{n_i}(M)$ 
\begin{equation} \label{cover}
(V \times \Pi_i C_{n_i}(U_i)) \times_{\Pi_i C_{n_i}(M)} C_m(M) 
\end{equation}
Under generic trasversality conditions of the map $(\ref{map})$, the space (\ref{cover}) defines a manifold with corners. 
For different $V$ and $U_1, U_2$, (\ref{cover}) defines a covering of the manifold $ C_{m,n}(\alpha)$.

\subsection{Observables}


We now construct the observable $\mathcal{O}_{\alpha_0}$ in terms of Chern-Simons integrals.
In order to define these integrals we consider graphs of the following type. The graphs are allowed to have external edges. The vertices are or trivalents or univalents. Each univalent vertex is labeled by a component connected of $N$, the vertices in the same component connected are cyclically ordered. We also assume that each connected component has at least a univalent vertex.  

For such a graph $\gamma$, we denote by $V_u(\gamma)$ the set of univalent vertices and by $V_t(\gamma)$ the set of trivalent vertices. 







For each connected component $i$ of $N$ define the trace $\text{Tr}_i$ as follows. Let $(t_1^i, t_2^i, ..., t_n^i)$ be the coordinates in cyclic order of the univalent vertices of $\gamma$ in the component $i$. 
For $X_i \in \mathfrak{g}_{t_i}$ ($0 \leq i \leq n$) define
$$ \text{Tr}_i(X_1 \wedge X_2 \wedge .... \wedge X_n)=  \text{Tr} ( X_n \text{hol}|_{t_{n-1}^i}^{t_n^i} ... \text{hol}|_{t_1^i}^{t_2^i}  X_1   \text{hol}|_{t_n^i}^{t_1^i})  .$$ 
The composition of these traces and the trace defined in the previous sections gives 
$$ \text{Tr}_{V(\gamma)} = (\bigotimes_i \text{Tr}_i ) \otimes \text{Tr}_{V_t(\gamma)} .$$

To the graph $\gamma$ is associated a differential form 
$$
\omega_{\gamma}  =  \bigwedge_{e \in E^{in}(\gamma)} \pi^*_e  P_s \in  \Omega^*(C_{V(\gamma)}(\alpha_0)) \otimes  \mathfrak{g}_{V(\gamma)}
$$
To the family of maps (\ref{map}) is associated the observable
\begin{equation} \label{action}
\mathcal{O}_{\alpha_0} =  \sum_{\gamma} \frac{1}{\text{Aut}(\gamma)} {\hbar}^{l(\gamma)}  \int_{C_{V(\gamma)}(\alpha_0)} \text{Tr}_{V(\gamma)} (\omega_{\gamma} \wedge \bigwedge_{e \in E^{ex}(\gamma)}  \pi_e^*( \psi)) 
\end{equation}
where $l(\gamma)= |E_{in}(\gamma)| - |V_t(\gamma)| $ is the number of loops of the graph $\gamma$. 


\subsection{Boundary-1}

The boundary faces of $V(\gamma)$ are associated to a subgraph that collapses to a point. 

Fix a subgraph $\delta$. Here we assume that each connected component of $\delta$ has at least a univalent vertex, and all the univalent vertices of $\delta$ are mapped into the same component of $N$. 
%

For each $x \in M$ we have the natural map that forget the trivalent vertices 
\begin{equation} \label{map1}
(C_{V(\delta)}(T_x M) / \sim )\rightarrow ( C_{V_u(\delta)}(T_x M) / \sim )
\end{equation}
A direction $d \in S(T_x M)$ defines an embedding $\R \rightarrow T_xM$, and it induces an embedding
\begin{equation} \label{map2}
 ( C_{V_u(\delta)}( \R ) / \sim ) \rightarrow ( C_{V_u(\delta)}(T_x M) / \sim ) 
\end{equation}
The fibered product of (\ref{map1}) and (\ref{map2})
\begin{equation} \label{degenerate}
\mathcal{S}_{\delta}^d= ( C_{V(\delta)}(T_x M)/ \sim ) \times_{( C_{V_u(\delta)}(T_x M) / \sim )}  ( C_{V_u(\delta)}( \R ) / \sim ) 
\end{equation}
is a manifold with corners of dimension $  3 |V_t(\delta)| + |V_u(\delta)| -2 $.
The union of the $\mathcal{S}_{\delta}^d$
$$ \mathcal{S}_{\delta}= \bigsqcup_{d \in S(TM)} \mathcal{S}_{\delta}^d $$
is a manifold with corners which is a fiber bundle  
\begin{equation} \label{fiberstring}
\pi_{\delta}: \mathcal{S}_{\delta} \rightarrow S(TM) .
\end{equation}

For each $e \in E(T)$, let $\pi_e : \mathcal{S}_{\delta} \rightarrow  \partial C_2(M)= C_2(TM)/ \sim $ be the projection in the configuration of the vertices of $e$.
To the graph $\delta$ we can associate the differential form 
\begin{equation} \label{degenerate-action}
\omega_{\delta}  =  \bigwedge_{e \in E^{in}(\delta)} \pi^*_e  (\eta + \phi) \in  \Omega^*(\mathcal{S}_{\delta}(\alpha)) \otimes  \mathfrak{g}_{V(\delta)}
\end{equation}

\begin{lemma} \label{wilsonface1}

Let $c_{\delta} $ be the integral of $\omega_{\delta}$ along the fibers of (\ref{fiberstring}):
$$c_{\delta} = (\pi_{\delta})_*   \omega_{\delta} \in  \Omega^* (S(TM)) \otimes  \mathfrak{g}_{V(\delta)}.$$

Let $k$ be the number of external edges of $\delta$. We have:
\begin{itemize}
\item If $k=0$, $c_{\delta}$ is a multiple of the differential form $\eta$.

\item If $k=2$, $c_{\delta}= 0$, unless
\begin{itemize}
\item $\delta$ has two univalent vertices and no trivalent vertices. To each vertex is attached an external edge.
\item $\delta$ has one trivalent vertex and one univalent vertex. The vertices are connected exactly by one edge. 
\end{itemize}

\item $c_{\delta}=0$ if $k \neq 0,2$.
\end{itemize}

\end{lemma}

\begin{proof}
The proof is analogous to the proof of Lemma \ref{face}.

Observe that the degree of the differential form $c_{\delta}$ is $2 - k$.
Moreover $c_{\delta}$ has to be an invariant polynomial of $\theta$ and $d \theta$. 
These two facts imply that $c_{\delta}$ is scalar multiple of $ \eta$ (if $k=0$) or a constant function (if $k=2$).
If $k=2$, as in Lemma \ref{face} the vanishing theorem of Kontsevich implies that $c_{\delta} = 0 $ unless $\delta$ has exactly two vertices.

\end{proof}

\subsection{Boundary-2}

We now consider the case where the univalent vertices of $\delta$ are subdivided in two sets $V_1, V_2$ corresponding to the two directions $d_1, d_2 \in S(T_xM)$. 
This case will be used to describe the boundary face that arises when two components of the link intersect or a component self-intersects (this will happen in a finite number of points on the interval $[0,1]$). 

Consider the following maps
\begin{equation} \label{map1-2}
(C_{V(\delta)}(T_x M) / \sim )\rightarrow  ( C_{V_1}(T_x M) / \sim ) \times ( C_{V_2}(T_x M) / \sim )
\end{equation}
\begin{equation} \label{map2-2}
 ( C_{V_1}( \R ) / \sim ) \times ( C_{V_2}( \R ) / \sim )  \rightarrow ( C_{V_1}(T_x M) / \sim ) \times ( C_{V_2}(T_x M) / \sim )
\end{equation}
As before the fibered product of (\ref{map1-2}) and (\ref{map2-2}) defines a manifold with corners $\mathcal{S}_{\delta}^{d_1,d_2}$.
The union over all the pairs $(d_1, d_2)$, 
$$ \mathcal{S}_{\delta}= \bigsqcup_{(d_1,d_2) \in S^2(TX)} \mathcal{S}_{\delta}^{d_1,d_2} $$
is a manifold with corners which is a fiber bundle  
\begin{equation} \label{fiberstring2}
\pi_{\delta}: \mathcal{S}_{\delta} \rightarrow S(TM)^2. 
\end{equation}

The differential form $\omega_{\delta} \in  \Omega^*(\mathcal{S}_{\delta}(\alpha)) \otimes  \mathfrak{g}_{V(\delta)}$ is defined analogously to (\ref{degenerate-action}).


\begin{lemma} \label{wilsonface-2}



Let $k$ be the number of external edges of $\delta$. 

\begin{itemize}

\item If $k=0$, $c_{\delta}$ is a constant function.
\item if $k \neq 0$, $c_{\delta}=0$. 

\end{itemize}

\end{lemma}
\begin{proof}


$c_{\delta}$ is a closed differential form of degree $-k$. Therefore if $c_{\delta} \neq 0$ we need to have $k=0$. These imply that $c_{\delta}$ is a constant function.

\end{proof}

\subsection{Invariance}


We want to extend Proposition \ref{homotopy} to the family of links (\ref{map}). 
Consider first the case without self-intersection of the link. As in Proposition \ref{homotopy}, we can consider a one parameter family of data. Using the extended propagator we can extend formula (\ref{action}) to define an element
$$ \mathcal{O}_{\alpha} \in \Omega^*([0,1]) \otimes  \mathcal{O}(H^*(M)) [\hbar] .$$ 

Let 
$$\alpha' : [0,1] \times N \rightarrow S([0,1] \times TM)$$ 
be the tangent direction of the link.

\begin{proposition} \label{wilson}
The following equation holds
\begin{equation} \label{observable2}
d \mathcal{O}_{\alpha} +  \hbar \Delta \mathcal{O}_{\alpha} + \{\tilde{S}, \mathcal{O}_{\alpha} \} + \beta_1(\hbar) \left( \int_N (\alpha')^*(\tilde{\eta}) \right) \wedge \mathcal{O}_{\alpha} = 0 
\end{equation}
where $\beta_1(\hbar)$ is a universal formal series in $\hbar$ that is independent of $M$ and $\alpha$.
\end{proposition}
\begin{proof}


Consider the boundary face of $C_{V(\gamma)}(\alpha)$ associated to a subgraph $\delta$. 

If $\delta$ has an external edge the boundary face does not contribute using the same argument as Proposition \ref{homotopy}. Therefore we can assume that $\delta$ has at least a univalent vertex.
Let $\mathcal{S}_{\delta}$ be as in formula (\ref{fiberstring}).
Let $\gamma'= \gamma/ \delta$. The point of $S^1$ where the vertices of $\delta$ collapse defines a map $t_{\delta}:C_{V(\gamma')}(\alpha)  \rightarrow N$. 
The boundary face associated to $\delta$ is given by $ ( \alpha' \circ t_{\delta})^* ( \mathcal{S}_{\delta} )$ and is a fiber bundle
$$ ( \alpha' \circ t_{\delta})^* ( \mathcal{S}_{\delta} ) \rightarrow C_{V(\gamma')}(\alpha) .$$ 

The contribution of this face can be computed using Lemma \ref{wilsonface1}. The trace of the differential form $c_{\delta}$ is of the form $\beta_1(\hbar) \eta$ for some universal power series $\beta_1(\hbar)$. This yields the last term of (\ref{observable2}).    


\end{proof}

Consider now the more general case where the family of links $(\ref{map})$ can have self intersections. We want to understand how $\mathcal{O}_{\alpha_s}$ changes in this process.

Let $s_0 \in [0,1]$ be a point where $\alpha_{s_0}$ self-intersects. $\alpha_{s_0}$ has two special points that are mapped in the same point of $M$. Define $\hat{\mathcal{O}}_{\alpha_{s_0}}$ as in formula (\ref{action}) except that to all the graphs it is added a component connected given by a special edge joining the two special points. In formula (\ref{action}), for the special edge instead to use the propagator we use $ I_{\mathfrak{g}} $.

\begin{proposition} \label{wall}
The discontinuity of $\mathcal{O}_{\alpha_s}$ in $s_0$ is given by
$$ \lim_{s \rightarrow s_0^+} \mathcal{O}_{\alpha_s} - \lim_{s \rightarrow s_0^-} \mathcal{O}_{\alpha_s} = \pm \beta_2(\hbar) \hat{\mathcal{O}}_{\alpha_{s_0}} $$
for some universal power series $ \beta_2(\hbar)$. The sign is equal to the sign of the crossing.
\end{proposition}
\begin{proof} 
The discontinuity is due to the boundary face associated to the graphs that collapse at the self-intersection point. The contribution of this boundary face can be computed as in Proposition \ref{wilson} using Lemma \ref{wilsonface-2}. $\beta_2(\hbar)$ is defined as the trace of $c_{\delta}$. 

\end{proof}

\subsection{Framed Links}


A frame of a link $\alpha_0$ is the choice of a normal vector field on the link which is non-vanishing
everywhere on the link. Two frames are considered equivalent if they are homotopic.

To a frame $f_0$ for $\alpha_0$ we can associate a submanifold with boundaries $\mathcal{F}^{f_0}$ of $(\alpha_0')^*(S(TM))$: 
$$\mathcal{F}^{f_0} = \bigsqcup_{t \in N} \{ \cos(\theta) \alpha_0'(t)  + \sin(\theta) f_0 (t) | 0 \leq \theta \leq \pi \} .$$ 
The natural projection $\mathcal{F}^{f_0} \rightarrow  N$ is a fibration whose fibers are half-circles. 
To the framed link $(\alpha_0,f_0)$ we associate the observable
$$ \mathcal{O}_{\alpha_0}^{f_0} =  \text{exp}{ \left( \frac{\beta_1(\hbar)}{2} \int_{\mathcal{F}^{f_0}} \eta \right) }  \mathcal{O}_{\alpha_0}.$$


The construction above can be easily extended to a family of links. To a family of frames $f$ for a family of links (\ref{map}) is associated the manifold with corners $\mathcal{F}^f$: 
$$\mathcal{F}^f = \bigsqcup_{(s,t) \in [0,1] \times N} \{ \cos(\theta) \alpha_s(t)  + \sin(\theta) f_s (t) | 0 \leq \theta \leq \pi \} . $$  
This is a fibration $\mathcal{F}^f \rightarrow [0,1] \times N$ with fibers half-circles. $\mathcal{O}_{\alpha}^f$ is defined as  
$$ \mathcal{O}_{\alpha}^f =  \text{exp}{ \left( \frac{\beta_1(\hbar)}{2} \int_{\mathcal{F}^{f_s}} \eta \right) }  \mathcal{O}_{\alpha}.$$

\begin{proposition}
Let $(\alpha_0,f_0)$ and $(\alpha_1,f_1)$ be two framed links, and let $(\alpha,f)$ be a framed family of links connecting $(\alpha_0,f_0)$ and $(\alpha_1,f_1)$. If $\alpha$ has not self-intersections the following equation holds 
$$ d \mathcal{O}_{\alpha}^f +  \hbar \Delta \mathcal{O}_{\alpha}^f + \{S, \mathcal{O}_{\alpha}^f \} = 0 .$$

In a point $s_0 \in [0,1]$ where $\alpha$ self-intersects, $\mathcal{O}_{\alpha_s}^{f_s}$ jumps of $ \pm \beta_2(\hbar)  \hat{\mathcal{O}}_{\alpha_{s_0}}^{f_{s_0}}$.

\end{proposition}
\begin{proof}

The first part of the proposition follows from formula (\ref{observable2}) and the following application of the Stokes Theorem
$$ d \int_{\mathcal{F}^{f_s}} \eta =  2 \int_N (\alpha')^*(\eta) .$$ 
The second part follows from Proposition (\ref{wall}).


\end{proof}





\begin{thebibliography}{10}

\bibitem{AS} S. Axelrod, I. M. Singer, \emph{Chern-Simons perturbation theory}, Proceedings of the XXth International Conference on Differential Geometric Methods in Theoretical Physics, Vol. 1, 2 (New York, 1991),  3--45, World Sci. Publ., River Edge, NJ, 1992. \emph{Chern-Simons perturbation theory II}  J. Differential Geom.  39  (1994),  no. 1, 173--213.

\bibitem{Co} K. J. Costello, \emph{Renormalisation and the Batalin-Vilkovisky formalism}, arXiv:0706.1533.

\bibitem{BC} R. Bott; A. Cattaneo, \emph{Integral invariants of $3$-manifolds}, J. Differential Geom. 48 (1998), no. 1, 91--133. \emph{Integral invariants of 3-manifolds. II}, J. Differential Geom. 53 (1999), no.1, 1--13. 

\bibitem{CM} A. Cattaneo; P. Mnev, \emph{Remarks on Chern-Simons invariants}, preprint.


\bibitem{Ko} M. Kontsevich, \emph{Feynman diagrams and low-dimensional topology}, First European Congress of Mathematics, Vol. II (Paris, 1992), 97--121, Progr. Math., 120, Birkhäuser, Basel, 1994. 



\end{thebibliography}
\end{document}